\documentclass[11pt,a4paper]{article}
\usepackage[latin1]{inputenc}
\usepackage{amsmath}
\usepackage{amsfonts}
\usepackage{amssymb}
\usepackage{amsthm}
\usepackage{bbm}
\usepackage{color}
\usepackage{dsfont}
\usepackage{slashed}
\usepackage{ogonek}

\usepackage[all]{xy}
\makeindex
\newtheorem{ev}{Everything}[section]
\newtheorem{theorem}[ev]{Theorem}
\newtheorem{lemma}[ev]{Lemma}

\newtheorem{prop}[ev]{Proposition}
\theoremstyle{definition}
\newtheorem{exa}[ev]{Example}
\newtheorem{mydef}[ev]{Definition}
\theoremstyle{remark}
\newtheorem{rema}[ev]{Remark}
\makeatletter
\newtheoremstyle{indented}
  {1pt}
  {1pt}
  {\addtolength{\@totalleftmargin}{1 em}
   \addtolength{\linewidth}{-4 em}
   \parshape 1  1.5 em \linewidth}
  {}
  {\bfseries}
  {.}
  {.5em}
  {}
  \makeatother
 \theoremstyle{indented}
 
\newtheorem{ith}[ev]{Theorem} 

\begin{document}

\begin{center}
{\bf DIFFERENTIAL CALCULUS}
\vspace{0,5cm}
{\bf ON JORDAN ALGEBRAS AND JORDAN MODULES}
\vspace{1cm}

Alessandro CAROTENUTO,
Ludwik D\k ABROWSKI
\vspace{0,2cm}

\textit{\small Scuola Internazionale Superiore di Studi Avanzati (SISSA)}\\{\small  via Bonomea 265, I-34136 Trieste, Italy}\\
{\small acaroten@sissa.it}, {\small dabrowski@sissa.it}
\vspace{0,5cm}

Michel DUBOIS-VIOLETTE\\
\vspace{0,2cm}
\textit{\small Laboratoire de Physique Th\'{e}orique, CNRS, Universit\'{e} Paris-Sud,\\ Universit\'{e} Paris-Saclay}, {\small   B\^atiment 210, F-91405 Orsay, France}\\
{\small michel.dubois-violette@u-psud.fr}
 \end{center}

\begin{abstract}
\noindent
Having in mind applications to particle physics we develop the differential calculus over Jordan algebras and the theory of connections on Jordan modules. In particular  we focus on differential calculus over the exceptional Jordan algebra and provide a complete characterization of the theory of connections for free Jordan modules. 
\end{abstract}
\section{Introduction}
It is 
quite legitimate
to expect that the finite spectrum of fundamental particles of matter (fundamental fermions) corresponds to representations of some finite quantum space.
Such a virtual space should be described by its observables, i.e. a quantum analogue of some class of real 
functions over it.
This is of course at the core of noncommutative geometry, where C*-algebras correspond to the noncommutative analogues of algebras of continuous complex functions. This formalism, extended also to real  C*-algebras
and enriched with additional structure, has been in particular applied to the Standard Model (\cite{C1},\cite{C2},\cite{C3},\cite{C4},\cite{CM},\cite{vS}) as well as to other Higgs gauge models (\cite{MDV2},\cite{MDV3},\cite{DVKM}). Moreover quaternionic C*-algebras, seen as generalizations of the algebra of continuous quaternionic functions on a classical space, have been studied within this formalism.
\\However already at the early beginning of quantum theory it was pointed out that the appropriate algebraic structures for finite quantum systems are the finite-dimensional 
formally real Jordan algebras,
 since this is the right framework in which one has spectral theory and  the physical interpretation in terms of observables and states (\cite{JNW},\cite{GPR}).

The real vector space of self-adjoint elements of a C* algebra is a formally real Jordan algebra (it is in fact a JB-algebra which implies the formal reality and is equivalent to it in finite dimension, see e.g. \cite{McCrimmon} or \cite{Iordanescu}).
In fact Jordan subalgebras of this kind of algebras cover almost all the possible cases, with the only exception of the real Albert factor, which is the $27$-dimensional algebra of three by three hermitian matrices with octonionic entries  \cite{Albert}. 
In a recent work \cite{MDV1} it has been 
suggested that this exceptional algebra may play a key role in the description of the internal space of fundamental fermions in the Standard Model of particle physics and in particular, that the implicit triality underlining the exceptional algebra may be related to the three generations of fundamental fermions.

The aim of this work is to 
outset the representation theory of formally real (also called 
Euclidean) finite-dimensional
 finite-dimensional Jordan algebras. We investigate Jordan modules over Jordan algebras and elaborate differential calculus and the theory of connections on Jordan modules. From a physical point of view, this corresponds to develop gauge theories for a quantum theory in which one allows for Jordan algebras as algebras of observables. It is needless to say that the groups of automorphisms of the Jordan algebras play a fundamental role in this theory 
(for instance acting as gauge group). We also provide certain more general constructions in the setting of Jordan algebras, and also in a broader setting of noncommutative and nonassociative algebras.

\section{Jordan algebras and finite quantum spaces}
Here and in the following a Jordan algebra is meant to be unital and finite-dimensional if not differently specified. 
We will provide some definitions for the category of all algebras, thus the term algebra without further specification will generally denote a noncommutative and nonassociative algebra. Also every vector space is meant as vector space over the field of real numbers if not differently specified. We use the Einstein summation for repeated up-down indices.
\begin{mydef}
A \textit{Jordan algebra} is a vector space $J$ together with a bilinear product $\circ: J \times J\rightarrow J$ such that
\begin{equation}\label{JP1}
x \circ y= y\circ x
\end{equation}
and
\begin{equation}\label{JP2}
x\circ(y\circ x^2)=(x\circ y)\circ x^2
\end{equation}
for any $x,y \in J.$
\end{mydef}
Condition \eqref{JP2} is called Jordan identity and it is equivalent to
\begin{equation}
\left[L_{x\circ y},L_z\right] +\left[L_{z\circ x},L_y\right]+\left[L_{y \circ z},L_x\right]=0
\end{equation}
where $L_x$ denotes the (left) multiplication by $x \in J.$ 
\\We provide two examples of Jordan algebras.
\begin{exa}\label{sja}
Let $(A,\;)$ be an associative algebra, we define the Jordan algebra $A^+=\left(A,\circ\right)$ to be the vector space $A$ equipped with the product given by
\begin{equation}
x \circ y = \frac{1}{2} \left(x y+y x\right)
\end{equation}
for all $x,y\in A.$
 One verifies by direct check of properties \eqref{JP1} and \eqref{JP2} that $A^+$ is a Jordan algebra. Any Jordan algebra which is isomorphic to a Jordan subalgebra of a Jordan algebra of this kind is called a \textit{special Jordan algebra.}
\\In particular if the associative algebra $A$ is endowed with an involution $*:A\rightarrow A$ that is
\begin{eqnarray*}
&(x^*)^*=x
\\&(xy)^*=y^*x^*
\end{eqnarray*}
then the subspace $A_{sa}=\{a \in A\mid a^*=a\}$ of self-adjoint elements in $A$ is not a subalgebra of $(A,\;)$ but it is a Jordan subalgebra of the special Jordan algebra $A^+$ and it is therefore special $\diamond$
\end{exa}
\begin{exa}
The exceptional Jordan algebra $(J^8_3, \circ)$  is defined as follows: its elements are $3\times 3$ hermitian matrices with octonionic entries
\begin{eqnarray*}
J^8_3=\{x\in M_3(\mathbb{O})\mid x=x^*\}
\end{eqnarray*}
and the product $\circ$ is given by the anticommmutator
\begin{equation}
x \circ y = \frac{1}{2} \left(x y+y x\right)
\end{equation}
for any $x,y \in J^8_3.$
It is a classical result (\cite{Albert}) that $J^8_3$ is a Jordan algebra which is not a special one. $\diamond$
\end{exa}
In the following we shall write $xy$ for the product of two elements in Jordan algebras (and in other kind of algebras) when no confusion arises.
\begin{mydef}
An \textit{Euclidean} (or \textit{formally real}) Jordan algebra is a real Jordan algebra $J$ satisfying the \textit{formal reality condition}
\begin{equation}
x^2+y^2=0 \Leftrightarrow x=y=0
\end{equation}
for any $x,y\in J.$
\end{mydef}
 Any Euclidean finite-dimensional Jordan algebra $J$ has a unit, moreover if $x\in J$ there is a spectral resolution of $x$ with real eigenvalues (see \cite{MDV1} for more details).  
\\The above examples are quite exhaustive in view of the following classical theorem (\cite{JNW}).
\begin{theorem}\label{Jsum}
Any finite-dimensional Euclidean Jordan algebra is a finite direct sum of simple Euclidean finite-dimensional Jordan algebras. Any finite-dimensional simple Euclidean Jordan algebra is isomorphic to one of the following:
\begin{equation*}
\begin{split}
& \mathbb{R},  \; JSpin_{n+2}=J_2^{n+1}=\mathbb{R}\oplus \mathbb{R}^{n+2},
\\ &J^1_{n+3}, \; J^2_{n+3}, \; J^4_{n+3}, \; J^8_3 
\end{split}
\end{equation*}
for $n\in \mathbb{N}.$
\end{theorem}
In the above statement $J^8_3$ is the only non special Jordan algebra, while $J^1_{n},J^2_{n},J^4_{n}$ denote $n \times n$ hermitian matrices with real, complex and quaternionic entries respectively, with product given by  the anticommutator. $JSpin_{n}=\mathbb{R}\oplus \mathbb{R}^n$ are the \textit{spin factors} equipped with the product
\begin{equation}
(s \oplus v)(s' \oplus v')=(ss' + \langle v, v' \rangle )\oplus(sv'+ s'v)
\end{equation}
where $\langle \cdot, \cdot \rangle$ denotes the Euclidean scalar product on $\mathbb{R}^n.$ The spin factor $JSpin_{1}$ is absent from this list since it is isomorphic to the commutative and associative algebra $\mathbb{R}^2$ which is not simple. The following isomorphisms hold:
\begin{equation*}
J^1_{1}=J^2_{1}=J^4_{1}=J^8_1=\mathbb{R} 
\end{equation*}
and
\begin{equation*}
J^1_{2}= JSpin_{2}, \;   J^2_{2}=JSpin_{3}, \; J^4_2= JSpin_{5}, \; J^8_2= JSpin_{9}
\end{equation*}
while $J^8_n$ is not a Jordan algebra  for $n\geq 4.$
This list gives all finite-dimensional Jordan algebras corresponding to finite quantum spaces.
\section{Center and derivations}
\begin{mydef}
Let $A$ be an algebra, define the \textit{associator} by
\begin{equation}
[x,y,z]=(xy)z-x(yz)
\end{equation}
for any $x,y,z \in A.$
The \textit{center} of $A,$ denoted by $Z(A)$, is the associative and commutative subalgebra of elements $z\in A$ satisfying 
\begin{equation}\label{center}
[x,z]=0, \; [x,y,z]=[x,z,y]=[z,x,y]=0
\end{equation}
for any $ x,y \in A.$
\end{mydef}
One has the following result.
\begin{prop}
Let $A$ be a commutative algebra and let $z \in A,$ then $z\in Z(A)$ if and only if 
\begin{equation}
[x,y,z]=0
\end{equation}
for all $x,y \in A.$
\end{prop}
\begin{proof}
The condition $[x,z]=0$ is trivial for any $ x,z \in A$ since we have taken $A$ commutative. If the condition $[x,y,z]=0$ holds, then for every $x,y \in A$ one has:
\begin{equation}
0=[x,y,z]-[y,x,z]=[y,z,x]
\end{equation}
and
\begin{equation}
0=-[x,y,z]=[z,x,y]
\end{equation}
for any $x,y \in A,$ in view of the commutativity.
\end{proof}
In particular, the proposition above is valid for all Jordan algebras.
\begin{mydef}
A \textit{derivation} of an algebra $A$ is a linear endomorphism $X$ of $A,$ such that one has
\begin{equation}
X(xy)= X(x)y +xX(y)
\end{equation}
for all $x,y \in A.$
\end{mydef}
\begin{prop}
The vector space $Der(A)$ of derivations of an algebra $A$ has the following properties:
\begin{enumerate}
\item $Der(A)$ is a Lie algebra with respect to the commutator of endomorphisms.
\item $Der(A)$ is a module over the center $Z(A).$
\item The center of $A$ is stable with respect to derivations, that is $X(z)\in Z(A)$ for all $X\in Der(A)$ and   for any $z \in Z(A).$
\item The following formula holds:
\begin{equation}
[X_1,zX_2]=X_1(z)X_2+z[X_1,X_2]
\end{equation}
for all $X_1, X_2 \in A$ and $z \in  Z(A).$
\end{enumerate}
\end{prop}
\begin{proof}
$(1),$ $(2)$ and $(4)$ are trivial, we have only to prove stability of the center. Let $z \in Z(A)$ and $ X\in Der(A),$ we have:
\begin{equation}
\begin{split}
&[x,y,X(z)]=(xy)X(z)-x(yX(z))=
\\&=X\left((xy)z\right)-X(xy)z-\left(xX(yz)-x\left(X(y)z\right)\right)=
\\&=X\left((xy)z\right)-X(xy)z-X(x(yz))+X(x)(yz)+x\left(X(y)z\right)=
\\&=X\left([x,y,z]\right)-[X(x),y,z]-[x,X(y),z]=0
\end{split}
\end{equation}
for any $x, y \in A.$ Similarly one proves that $\left[x,X(z),y\right]=\left[X(z),x,y\right]=0$ and $\left[x,X(z)\right]=0.$ 
\end{proof}
Thus the pair $\left(Z(A),Der(A)\right)$ form a Lie-Rinehart algebra (\cite{Rin}, \cite{Hue}).
\\For Jordan algebras, the list of derivations for the finite-dimensional non-exceptional simple Euclidean Jordan algebras covers the list of the non exceptional simple Lie algebra, i.e. the Lie algebras denoted by $\mathfrak{a}_n,$ $\mathfrak{b}_n,$ $\mathfrak{c}_n$ and $\mathfrak{d}_n$ in the Cartan classification, while for the exceptional Jordan algebra $J^8_3$ the algebra of derivations is given by the exceptional Lie algebra $\mathfrak{f}_4$ as shown in the following example.
\begin{exa}
As just mentioned, the Lie algebra of derivations of the exceptional Jordan algebra $J^3_8$ is the exceptional Lie algebra $\mathfrak{f}_4$ (see e.g. \cite{Yokota}).
Introduce the standard basis of the exceptional Jordan algebra
\begin{equation*}
\begin{split}
&E_1=\left(\begin{matrix}1 & 0 & 0 \\0 & 0 & 0  \\0 & 0 & 0 \end{matrix} \right), \;
E_2=\left(\begin{matrix}0 & 0 & 0 \\0 & 1 & 0  \\0 & 0 & 0 \end{matrix} \right), \;
E_3=\left(\begin{matrix}0 & 0 & 0 \\0 & 0 & 0  \\0 & 0 & 1 \end{matrix} \right)
\\&F_1^j=\left(\begin{matrix}0 & 0 & 0 \\0 & 0 & \epsilon_j  \\0 & \overline{\epsilon}_j & 0 \end{matrix} \right), \;
F_2^j=\left(\begin{matrix}0 & 0 & \overline{\epsilon}_j\\0 & 0 & 0  \\\epsilon_j & 0 & 0 \end{matrix} \right), \;
F_3^j=\left(\begin{matrix}0 & \epsilon_j & 0\\\overline{\epsilon}_j & 0 & 0  \\0 & 0 & 0 \end{matrix} \right) \;
\\
\end{split}
\end{equation*}
where $\epsilon_j \; j \in \{0,...,7\}$ are a basis of the octonions, so $\epsilon_0=1,$ $\epsilon^2_j=-1$ for $j\neq 0$ and the multiplication table of octonions holds (see e.g. on \cite{Baez}). 
\\As vector space, $\mathfrak{f}_4$ admits a decomposition
\begin{equation*}
\mathfrak{f}_4= \mathfrak{D}_4 \oplus \mathfrak{M}^-
\end{equation*}
given as follows.
$\mathfrak{D}_4$ is the subspace of derivations which annihilates the diagonal of any element in $J^8_3,$ that is
\begin{equation*}
 \delta E_i=0 \quad i\in \{1,2,3\} 
\end{equation*}
for any $\delta \in \mathfrak{D}_4.$ 
An interesting and concrete characterization of $\mathfrak{D}_4$ is given by the following theorem (see e.g. chapter $2$ of \cite{Yokota}).
\begin{ith}
The algebra $\mathfrak{D}_4$ is isomorphic to $\mathfrak{so}(8)=\mathfrak{d}_4.$ The isomorphism is given via the equality:
\begin{equation*}\label{so8}
\delta \left(\begin{matrix}
\xi_1 & x_3 & \overline{x}_2 
\\\overline{x}_3 & \xi_2 & x_1 
\\ x_2 & \overline{x}_1 & \xi_3
\end{matrix} \right)= \left(\begin{matrix}
0 & D_3x_3 & \overline{D_2x}_2 
\\\overline{D_3x}_3 & 0 & D_1x_1 
\\ D_2x_2 & \overline{D_1x}_1 & 0
\end{matrix} \right)
\end{equation*}
where $ \delta \in \mathfrak{D}_4$ and $D_1, D_2, D_3 \in \mathfrak{so}(8).$ $D_2,D_3$ are determined by $D_1$ from the principle of infinitesimal triality
\begin{equation}\label{infinitesimal triality}
(D_1x)y+ x(D_2)y= \overline{D_3(\overline{xy})} 
\end{equation}
for any $x, y \in \mathbb{O}.$
\end{ith}
Elements of the vector space $\mathfrak{M}^-$ are $3\times 3$ antihermitian octonion matrices with every element on the diagonal equal to zero.
Every $M \in \mathfrak{M}^-$ defines a linear endomorphism $\tilde{M}:J^8_3\rightarrow J^8_3,$ via the commutator
\begin{equation*}
\tilde{M}(x)=Mx-xM
\end{equation*}
where in the expression above juxtaposition is understood as the usual raw by column matrix product. $\diamond$ 
\end{exa}

The following classical result about derivations of Jordan algebras, due to Jacobson (\cite{Jacobson2}) and Harris \cite{Harris}, is the equivalent of Witehead's first lemma for Lie algebras.
\begin{theorem}\label{Jacobsonder}
Let $J$ be a finite-dimensional semi-simple Jordan algebra, let $X \in Der(J).$ There exists a finite number of couples of elements $x_i,y_i\in J$ such that one has
\begin{equation}
X(z)=\sum [x_i,z,y_i]
\end{equation}  
for any $z \in J.$
\end{theorem} 
\section{Jordan modules}
The familiar definition of bimodules over associative algebras is not suitable for nonassociative algebras such as Jordan algebras. Indeed, due to nonassociativity, such a definition would imply that a Jordan algebra is not a module over itself if one takes the multiplication as action of the algebra. A more correct definition is the following (\cite{Ei}, \cite{Jacobson1}, see also \cite{MDV3},\cite{KOS}):
\begin{mydef}\label{dm}
Let $J$ be a Jordan algebra, a \textit{Jordan bimodule} over $J$ is a vector space $M$ together with two bilinear maps
\begin{eqnarray*}
J\otimes M \rightarrow M \quad x \otimes m \mapsto xm
\\M\otimes J \rightarrow M\quad m \otimes x \mapsto mx 
\end{eqnarray*} 
such that $J\oplus M,$ endowed with the product 
\begin{equation}
(x,m)(x',m')=\left(xx',xm'+mx'\right)
\end{equation}
 is a Jordan algebra by itself. 
\end{mydef}
This definition is equivalent to require the following properties of the action of the Jordan algebra $J$ on its module $M:$
\begin{equation}\label{pmj}
\begin{cases}
mx= xm
\\x(x^2m)=x^2(mx)
\\(x^2y)m-x^2(ym)= 2((xy)(x m )- x(yxm)) 
\\\mathds{1}_J m=m
\end{cases}
\end{equation}
far any $x,y \in J$ and $m\in M.$
\\Notice that from the first of relations above one has not to specify if using left or right multiplication so we shall call Jordan module any bimodule over a Jordan algebra.
The second relation can also be written as
\begin{equation}
\left[L_x,L_{x^2}\right]=0
\end{equation}
while the third is written as
\begin{equation}
L_{x^2y}-L_{x^2}L_{y}-2L_{xy}L_{x}+2L_xL_{y}L_x=0
\end{equation}
which is equivalent to the conditions
\begin{equation}
\begin{cases}
L_{x^3}-3L_{x^2}L_x+2L^3_x=0
\\ \left[\left[L_x,L_{y}\right],L_{z}\right]+L_{\left[x,y,z\right]}=0
\end{cases}
\end{equation}
for every $x,y,z \in J,$ where here $L_x$ denote the multiplication by $x \in J$ in $M.$
\begin{exa}
It follows from definition \eqref{dm} that any Jordan algebra $J$ is a module over itself. More generally, let $J$ be a finite-dimensional Jordan algebra, a free $J$-module $M$ is of the form
\begin{equation*}
M=J \otimes E
\end{equation*}
where $E$ is a finite-dimensional vector space and the action of $J$ on $M$ is given by multiplication on the first component of $M.$ It turns out that, when $J$ is the exceptional Jordan algebra, any finite module over $J$ is a free module \cite{Jacobson1}. $\diamond$
\end{exa}
\begin{exa}
Let $A$ be an associative algebra, let $J\subseteq A^+$ be a special Jordan algebra as in example \eqref{sja}. Any element $x\in J$ is also an element of $A$ and $A$ is endowed with $J$-module structure by setting 

\begin{equation}
L_x a= x\circ a= \frac{1}{2}(ax+xa)
\end{equation}
for any $x\in J$ and $a\in A.$ In the two following examples the same construction is explicitely given for the antihermitian real, complex and quaternionic matrices as module over hermitian matrices and for the Clifford algebras $Cl(\mathbb{R}^n)$ as modules over the spin factors $JSpin_n.$
\end{exa}
\begin{exa}\label{ah}
Denote by $A^i_n (i=1,2,4)$ the vector space of antihermitian matrices with real, complex and quaternionic entries respectively. $A^i_n$ is a module over the special Jordan algebra $J^i_n$ with action given by the matrix anticommutator:
\begin{equation}\label{ahonh}
L_x a= x\circ a= \frac{1}{2}(ax+xa) 
\end{equation} 
for any $x \in J^i_n$ and $a \in A^i_n.$
Moreover, taking $J^i_n$ as free module over itself we have:
\begin{equation*}
J^i_n\oplus A^i_n= M^i_n
\end{equation*}
which is the $J$-module of $n\times n$ real, complex or quaternionic matrices with action of $J$ defined as above by \eqref{ahonh}. $\diamond$
\end{exa}
\begin{exa} \label{cliff}
The Clifford algebra
\begin{equation*}
Cl\left(\mathbb{R}^n\right)=\frac{T\left(\mathbb{R}^n\right)}{\left(\{x\otimes x=||x||^2,\; \forall x\in \mathbb{R}^n\}\right)}
\end{equation*}
is a module over the Jordan algebra $JSpin_{n}=\mathbb{R}\oplus \mathbb{R}^n$ with action given by
\begin{equation}
L_x[y]=\frac{1}{2}\left([x\otimes y]+[y\otimes x]\right)
\end{equation}   
for any $x \in \mathbb{R}^n$ and $[y]\in Cl(\mathbb{R}^n).$
\end{exa}
	\begin{mydef}
Let $J$ be a Jordan algebra, let $M$ and $N$ be two modules over $J,$ then a \textit{module homomorphism} between $M$ and $N$ is a linear map $\varphi:M\rightarrow N$ such that
\begin{equation}\label{homomorphism}
x\varphi(m)=\varphi(xm) 
\end{equation}
for all $m \in M $ and $ x \in J.$
\end{mydef}

For homomorphisms between free modules over a fixed Jordan algebra, one has the following results.
\begin{theorem}\label{main}
Let $J$ be a finite-dimensional unital simple Jordan algebra, let $M=J\otimes E$ and $N=J\otimes F,$ where $E$ and $F$ are two finite-dimensional vector spaces, be free modules over $J.$ Then every module homomorphism $\varphi:M\rightarrow N$ is of the form
\begin{equation}
\varphi(x\otimes v) = x\otimes Av \quad x \in J, v \in E
\end{equation}
where $A: M\rightarrow N$ is a linear map.
\end{theorem}
\begin{proof}
For sake of simplicity, start by taking $M=N=J,$ then a module homomorphism is a linear map $f:J\rightarrow J$ such that:
\begin{equation}
x\varphi (y)=\varphi(xy)
\end{equation}
for any $x,y \in J.$
In particular:
\begin{equation}
\varphi(x)=x\varphi(1)=x A
\end{equation}
for some $A \in J$ such that $A=\varphi(1).$
\\Now, from definition of module homomorphism, we have:
\begin{eqnarray}
\varphi(xy)=(xy)A=x\phi(y)=x\left(yA\right)\Rightarrow \left[x,y,A\right]=0
\end{eqnarray}
for all $x,y \in J,$ hence $A\in Z(J).$ Thus  $A \in \mathbb{R},$ in view of simplicity of $J.$ 
\\More generally let $M=J\otimes E$ and $N=J\otimes F,$ denote as $e_\alpha$ and $f_\alpha$ a basis of $E$ and $F$ respectively. We have
\begin{equation}
\varphi(1\otimes e_\alpha)= A_\alpha^\lambda \otimes f_\lambda
\end{equation}
for some $A_\alpha^\lambda \in J.$ With the same argument as above, we get:
\begin{eqnarray}
&\varphi(xy\otimes e_\alpha)=(xy)\varphi(1\otimes e_\alpha)=(xy)A_\alpha^\lambda \otimes f_\lambda
\\ &\varphi(xy\otimes e_\alpha)=(x\varphi(y\otimes e_\alpha))=x(y\varphi(1\otimes e_\alpha))=x(yA_\alpha^\lambda ) \otimes f_\lambda
\end{eqnarray}
and so every $A_\alpha^\lambda\in Z(J)$ and it is a real number. Using properties of tensor product we have
\begin{equation}
A_\alpha^\lambda \otimes f_\lambda = 1 \otimes  A_\alpha^\lambda f_\lambda
\end{equation}
and the statement follows by taking as map $A$ from $E$ into $F$ the linear transformation defined by $A(e_\alpha)=   A_\alpha^\lambda f_\lambda.$
\end{proof}
If the Jordan algebra $J$ is not simple the above theorem is generalized as follows:
\begin{lemma}
Let $J$ be a finite-dimensional unital Jordan algebra, let $M=J\otimes E$ and $N=J\otimes F$ be free modules over $J,$ with $E,F$ finite-dimensional vector space of dimension $m$ and $n$ respectively. Then if $f:M\rightarrow N$ is homomorphism of $J$ modules, there exist $\alpha_k \in Z(J)$ and $f_k \in M_{m\times n}$ such that:
\begin{equation}
f(1\otimes e)= \sum_k \alpha_k \otimes f_k(e)
\end{equation}
  for any $e \in E.$
\end{lemma}
From the above lemma we deduce the following result.
\begin{theorem} Let $J$ be a finite-dimensional unital Jordan algebra with center $Z(J)$.  Denote as $FMod_J$ the category of free Jordan modules over $J$ with homomorphisms of Jordan modules and as $FMod_{Z(J)}$ the category of free modules over the associative algebra $Z(J)$ with homomorphisms of modules over associative algebras. Then the following functor is an isomorphism of categories:
\begin{equation}
\begin{split}
\mathcal{F}: \quad  & J \otimes E \mapsto Z(J)\otimes E
\\ \quad& (\varphi:J \otimes E \rightarrow J \otimes F) \mapsto (\varphi_{Z(J)}: Z(J) \otimes E \rightarrow Z(J) \otimes F)
 \end{split}
\end{equation}
where $\varphi_{Z(J)}$ is the restriction of $\varphi$ to $Z(J) \otimes E.$
\end{theorem}
\begin{proof}
We begin by checking functoriality of $\mathcal{F}.$ Of course the image of the identity of $FMod_J$ is the identity of $FMod_{Z(J)}.$ Let
$\varphi:J \otimes E \rightarrow J \otimes F$ and $\phi:J \otimes F\rightarrow J \otimes H$ be two homomorphisms of free modules over $J$, then we have:
\begin{equation}
\mathcal{F}(\phi\circ\varphi)=(\phi\circ\varphi)_{Z(J)}
\end{equation}
From theorem \eqref{main} we know that $\varphi(Z(J) \otimes E) \subseteq Z(J) \otimes F,$ and so:
\begin{equation}
\mathcal{F}(\phi\circ\varphi)=(\phi\circ\varphi)_{Z(J)}=\phi_{Z(J)}\circ \varphi_{Z(J)}= \mathcal{F} (\phi) \circ \mathcal{F} (\varphi)
\end{equation}
which proves that $\mathcal{F}$ is a functor. 
Define $\mathcal{F}^{-1}$ as:
\begin{equation}
\begin{split}
\mathcal{F}^{-1}: \quad  & Z(J) \otimes E \rightarrow J\otimes E
\\ \quad & (\varphi: Z(J) \otimes E \rightarrow Z(J) \otimes F ) \mapsto (\varphi_J :J \otimes E \rightarrow J \otimes F)
 \end{split}
\end{equation}
where $\varphi_J$ is defined by regarding the elements of $Z(J)$ as elements of $J$ and setting:
\begin{equation}
\varphi_J(x\otimes e):= x \varphi (1\otimes e)
\end{equation}
for any $x \in J$ and $e \in E.$
\end{proof} 

Finally let us introduce the following notion that will be useful in the next section.
\begin{mydef}
Let $J$ be a Jordan algebra, let $M$ be a module over $J.$ A \textit{derivation} $d$ of $J$ into $M$ is a linear map $d:J \rightarrow M$ such that:
\begin{equation}
d(xy)= d(x)y+ x d(y) 
\end{equation}
for any $x, y \in J.$
\end{mydef} 
\section{Differential calculi}
Let us recall the following standard ``super version" of Jordan algebras (see e.g. in \cite{Kac}).
\begin{mydef}
A \textit{Jordan superalgebra} $\Omega=\Omega^0 \oplus \Omega^1$ is a $\mathbb{Z}_2$-graded vector space with a graded commutative product, meaning:
\begin{equation*}
xy= (-1)^{\mid x \mid \mid y\mid} yx
\end{equation*}
for all $x, y \in \Omega$ and such that this product respects the Jordan identity.
\end{mydef}
For a Jordan superalgebra it holds:
\begin{equation}
\left[x,y,z\right]= (-1)^{\mid y \mid \mid z \mid} \left[z,y,x\right]
\end{equation}
for all $x,y,z \in \Omega.$ If we introduce the graded commutator of $x$ and $y$ as
\begin{equation}
[x,y]_{gr}=xy  +({-1})^{\mid x \mid \mid y\mid} yx
\end{equation}
the Jordan identity is equivalent to:
\begin{equation}
\begin{split}
&(-1)^{\mid x \mid \mid z \mid}\left[L_{xy},L_z\right]_{gr}+ (-1)^{\mid z \mid \mid y \mid}\left[L_{zx},L_y\right]_{gr}+
\\&+(-1)^{\mid y \mid \mid x \mid}\left[L_{yz},L_x\right]_{gr} =0 
\end{split}
\end{equation}
for all $x,y,z \in \Omega.$ In what follows, we will deal with \textit{$\mathbb{N}$-graded Jordan superalgebras}, that means we are going to consider $\mathbb{N}$-graded  algebras $\Omega=\oplus_{\mathbb{N}}\Omega^n$ that are also Jordan superalgebras with respect to the $\mathbb{Z}_2$-grading induced by the decomposition in even and odd parts, that we shall denote respectively as $\Omega^+$ and $\Omega^-.$ 
\begin{mydef}
A \textit{differential graded Jordan algebra} is an $\mathbb{N}$-graded Jordan superalgebra $\Omega$ equipped with a \textit{differential}, which is a antiderivation $d$ of degree $1$ and with square zero, that is one has
\begin{eqnarray*}
d\Omega^n\subset \Omega^{n+1}
\\d^2=0
\end{eqnarray*}
and
\begin{equation*}
d(xy)=(dx)y+ (-1)^{\mid x \mid} x d(y)
\end{equation*}
for all $x,y \in \Omega.$ 
\end{mydef}
Such differential graded Jordan algebras are our models for generalizing differential forms, in particular when $\Omega^0=J$ we say that $\left(\Omega,d\right)$ is {\sl a differential calculus} over the Jordan algebra $J$ (this terminology is inspired from \cite{Woronowicz}).
A model of differential calculus over a Jordan algebra is the \textit{derivation-based differential calculus} which has been introduced in \cite{MDV1} and generalizes differential forms as defined in \cite{Koszul}.
\\Let us denote as $\Omega^1_{Der}(J)$ the $J$-module of $Z(J)$-homomorphisms from $Der(J)$ into $J.$ We define a derivation $d_{Der}:J \rightarrow \Omega^1_{Der} (J)$ by setting:
\begin{equation}
\left(d_{Der}x\right)(X):= X(x)
\end{equation}
\\for any $x \in J$ and $X\in Der(J).$
We refer to the pair $\left(\Omega^1_{Der}(J), d_{Der}\right)$ as the derivation-based first order differential calculus over $J.$
\\Let $\Omega^n_{Der}(J)$ be the $J$-module of $n$-$Z(J)$-linear antisymmetric mapping of $Der(J)$ into $J,$ that is any $\omega \in  \Omega^n_{Der}(J)$ is a $Z(J)$-linear mapping  $\omega: \wedge_{Z(J)}^n Der(J) \rightarrow J.$
\\Then $\Omega_{Der}(J)=\oplus_{n\geq 0} \Omega^n_{Der}(J),$  is an $\mathbb{N}$-graded Jordan superalgebra  with respect to wedge product of linear maps.  
One extends $d$ to a linear endomorphism of $\Omega_{Der}(J)$ by setting
\begin{equation*}
\begin{split}
&(d_{Der}\omega)(X_0,...,X_n)=\sum _{0 \leq k \leq n}(-1)^k X_k \left(\omega\left(X_0,...,\widehat{X_{k}},... X_n\right)\right)+
\\& +\sum _{0 \leq r <s \leq n}(-1)^{r+s}  \omega\left([X_r,X_s],X_0,...,\widehat{X_{r}},...,\widehat{X_{s}},... X_n\right) 
\end{split}
\end{equation*}
for any $\omega \in \Omega_n(J).$ This extension of
$d_{Der}$ is an antiderivation and $d^2_{Der}=0.$ Thus $\Omega_{Der}(J)$ endowed with $d_{Der}$ is a differential graded Jordan superalgebra with $\Omega_0=J.$ We refer to $\left(\Omega_{Der}(J), d_{Der}\right)$ as the \textit{derivation-based differential calculus over $J.$} 
\\In general, the derivation-based differential calculus does not play any privileged role in the theory of differential calculus over a given Jordan algebra. Howewer in the case of exceptional Jordan algebra $J^8_3,$ the derivation-based differential calculus is characterized up to isomorphism by the following universal property (\cite{MDV1}).
\begin{theorem} \label{dce}
Let $(\Omega,d)$ be a differential graded Jordan algebra and let $\phi: J^8_3\rightarrow \Omega^0$ be an homomorphism of unital Jordan algebras. Then $\phi$ has a unique extension $\tilde{\phi}: \Omega_{Der}\left(J^8_3\right)\rightarrow \Omega$ as homomorphism of differential graded Jordan algebras.
\end{theorem}
To prove this theorem we shall need the following result.
\begin{lemma}\label{lemmamio}
Let $\Gamma$ be a Jordan superalgebra, then $J^8_3\otimes\Gamma=\oplus_{n\in \mathbb{N}}J^8_3\otimes\Gamma^n$ is a Jordan superalgebra if and only if $\Gamma $ is an associative superalgebra.
\end{lemma}
This result is a consequence of the following lemmas proved in \cite{ZS} (Lemma $2$ and Lemma $3$ in \cite{ZS}).
\begin{lemma}\label{lZS1}
Let $\Gamma= \Gamma^+ \oplus \Gamma^-$ be a unital Jordan superalgebra whose even component $\Gamma^+$ is associative, then either one of the two equalities
\begin{equation}
[\Gamma^-, \Gamma^+, \Gamma^+]=0
\end{equation}
or
\begin{equation}
[\Gamma^-, \Gamma^+, \Gamma^+]=\Gamma^-
\end{equation}
holds.
\end{lemma}
\begin{lemma}\label{lZS2} 
Let $\Gamma$ be as above and such that $[\Gamma^-, \Gamma^+, \Gamma^+]=0,$ then either one of the two equalities
\begin{equation}
[\Gamma^+, \Gamma^-, \Gamma^-]=0
\end{equation}
or
\begin{equation}
[\Gamma^+, \Gamma^-, \Gamma^-]=\Gamma^+
\end{equation}
holds.
\end{lemma}
\begin{proof}[Proof of theorem \eqref{lemmamio}]
Let $\xi=\sum_{i} a_i \otimes b_i$ and $\eta= x \otimes y$ be elements in  $J^8_3 \otimes\Gamma$ we have to find whenever $\left[\xi^2,\eta,\xi\right]=0,$ that is
\begin{equation} \label{tens}
\begin{split}
&\left[\left(\sum_{i} a_i \otimes b_i\right)^2,x \otimes y,\sum_{j} a_j \otimes b_j\right]=
\\&\left[\sum_{i,j<i} a_i a_j \otimes b_i b_j + (-1)^{\mid b_i\mid \mid b_j\mid}a_i a_j \otimes b_i b_j ,x \otimes y,\sum_{k} a_k \otimes b_k\right]+
\\&+\left[\sum_{i} a_i^2\otimes b_i^2 ,x \otimes y,\sum_{k\neq i} a_k \otimes b_k\right]=0
\end{split}
\end{equation}
$\Gamma=\Gamma^+ \oplus \Gamma^-$ is a Jordan superalgebra and in particular $\Gamma^+$ is a graded subalgebra of $\Gamma$ and one knows (\cite{Wu}) that the algebra $J^8_3\otimes \Gamma^+$ is a Jordan graded algebra if and only if $\Gamma^+$ is associative. We must then assume $\Gamma^+$ associative. In expression \eqref{tens} let us take
\begin{equation}
\begin{split}
&\xi= a_{-1}\otimes 1 + a_0 \otimes e + \sum_{i} a_i \otimes o_i,
\\&\xi^2=  a^2_{-1}\otimes 1 + a^2_0 \otimes e^2+a_{-1}\otimes e + 
\\&\;+2\sum_{i}a_{-1} a_i \otimes o_i+ 2\sum_{i}a_{o} a_i \otimes \tilde{o}_i,
\\& \eta= x_0\otimes y_e+ x_1\otimes y_o
\end{split}
\end{equation}
where $a_i$ and $ x_i \in J^8_3, $ $e$ and $ y_e \in \Gamma^+,$ $o_i$ and $ y_o \in \Gamma^-$ and finally we set $\tilde{o}_i=eo_i\in \Gamma^- .$ 
Then one has
\begin{equation}
\begin{split}
&\left[\xi^2,\eta,\xi\right]=\left[a^2_0\otimes e^2,x \otimes y, a_j \otimes o_j\right]+ \left[ a_{-1}a_0\otimes e,x \otimes y, a_0 \otimes e+a_j \otimes o_j\right]+
\\&+\left[ a_{-1}a_i\otimes o_i,x \otimes y, a_0 \otimes e\right]+ \left[ a_{0}a_i\otimes \tilde{o}_i,x \otimes y, a_0 \otimes e\right] +
\\&+ \left[ a_{0}a_i\otimes \tilde{o}_i,x \otimes y, a_0a_j \otimes \tilde{o}_j\right]=0
\end{split}
\end{equation}
for all $a_i \in J^8_3.$ We can choose elements $a_i$'s and $x$ in $J^8_3$ in such a way that this condition is equivalent to
\begin{equation}
\begin{split}
&\left[e^2, y_e, o_j\right]+\left[e^2, y_o, o_j\right]+ \left[ e,y_o, e\right]+\left[e,y_e,o_j\right]+\left[e,y_o,o_j\right]+\left[o_i, y_e,  e\right]+
\\&+\left[o_i, y_o,  e\right]+ \left[ \tilde{o}_i, y_e, e\right] + \left[ \tilde{o}_i, y_o, e\right]  + \left[ \tilde{o}_i, y_e,\tilde{o}_j\right]+\left[ \tilde{o}_i, y_o,\tilde{o}_j\right]=0
\end{split}
\end{equation}
and varying elements in $\Gamma$ we see that condition above implies 
\begin{equation}
\left[\Gamma^-,\Gamma^+,\Gamma^+\right]+ \left[\Gamma^-,\Gamma^-,\Gamma^+\right]+\left[\Gamma^-,\Gamma^-,\Gamma^-\right]=0
\end{equation}
then, combining lemma $\eqref{lZS1}$ with lemma $\eqref{lZS2},$ we see that the equality above can hold only if all the summands above are identically zero, hence $\Gamma= \Gamma^+ \oplus \Gamma^-$ must be an associative superalgebra. 
\end{proof}
Now the proof of theorem \eqref{dce} is the same as in proposition $4$ of \cite{MDV1}, as we shall recall for sake of completeness.
\begin{proof}[Proof of theorem \eqref{dce}] For all $n \in \mathbb{N},$ $\Omega^n$ is a Jordan module over $J^8_3$ and from general theory of $J^8_3$ modules we know that every module over $J^8_3$ is a free module, hence we have
\begin{equation}
\Omega^n= J^8_3 \otimes \Gamma^n
\end{equation}  
where $\Gamma^n$ is a vector space. Any differential graded Jordan superalgebra over $J^8_3$ is then written as
\begin{equation*}
\Omega=\oplus_{n \in \mathbb{N}} J^8_3 \otimes \Gamma^n= J^8_3 \otimes \Gamma
\end{equation*} 
where $\Gamma= \oplus_{n \in \mathbb{N}} \Gamma^n$ is a Jordan superalgebra. 
 
 Consider the $J^8_3$-module $\Omega^1= J^8_3 \otimes \Gamma^1,$ and let $\{e^\alpha\}\subset \Gamma^1$ be a basis of $\Gamma^1.$ Let $\{\partial_k\}$ be a basis of $Der\left(J^8_3\right)$ with dual basis $\{\theta_k\}$ such that $\theta_k\left(\partial_j\right)=\delta_{kj}.$ We have
\begin{equation}
dx=\partial_k x\otimes c_\alpha^k e^\alpha
\end{equation}
for all $x \in J$ and for some real constants $c_\alpha^k$'s. Define the linear map $\tilde{\phi}$ from $\Omega^1_{Der}$ into $\Omega^1$ by
\begin{equation}
\tilde{\phi}\left(x \otimes \theta^k\right)= x \otimes c_\alpha^k e^\alpha.
\end{equation}
and extend it as homomorphism of superalgebras. We have $\tilde{\phi} \circ d_{Der}= d \circ \tilde{\phi},$ and uniqueness of $\tilde{\phi}$ follows from $d^2=0$ and the Leibniz rule.
\end{proof}
It is important to remark that this statement holds true only for the exceptional Jordan algebra and it is a direct consequence of the fact that the only irreducible module over $J^8_3$ is $J^8_3$ itself. 
\section{Connections and curvature for Jordan modules}
There are two equivalent definitions of derivation-based connections for modules of Jordan algebras and correspondingly two definitions of curvature. 
\begin{mydef} \label{pc}
Let $J$ be a Jordan algebra, a \textit{derivation-based connection} on a module $M$ over $J$ is a linear mapping $\nabla $ from $ Der(J)$ into the space of linear endomorphisms of the module $End(M),$ $\nabla: X \mapsto \nabla_X$ such that
\begin{equation}\label{prc}
\nabla_X (xm)=X(x)m+x\nabla_X m
\end{equation}
and
\begin{equation}\label{src}
\nabla_{zX} (m)=z\nabla_{X} (m)
\end{equation}
for any $x \in J,$ $m \in M$ and $z \in Z(J).$ 
\end{mydef}
From the first property it follows that if $\nabla$ and $\nabla'$ are two connections on the Jordan module $M$, then $\nabla_X-\nabla_{X}'$ is a $J$-module endomorphism.
\begin{mydef}
Let $\nabla$ be a derivation-based connection on a Jordan module $M.$ The \textit{curvature} of $\nabla$ is defined as
\begin{equation}
R_{X,Y}=[\nabla_X, \nabla_Y]-\nabla_{[X,Y]}
\end{equation}
for all $X,Y\in Der(J).$
\end{mydef}
It follows that $R_{X,Y}$ is a $J$-module endomorphism.
A connection will be called flat if its curvature is identically zero that is
\begin{equation}
R_{X,Y}(m)=0  
\end{equation}
for all $X, Y \in Der(J)$ and $m \in M.$
\begin{rema} 
In view of applications to particle physics, and in particular to Yang-Mills models, we are interested in classifying flat connections for Jordan modules. In fact, according to a standard heuristic argument ( see e.g. \cite{MDV2},\cite{DVKM}), any flat connection corresponds to a different ground state of the theory and the specification of the latter leads to different physical situations.
\end{rema} 
The second definition of derivation-based connections is more suitable to be generalized to connections not based on derivations.
\\ Let $J$ be a Jordan algebra, let $M$ be a module over $J$ and denote as $\Omega^n_{Der}(M)$ the $J$-module of all $n$-$Z(J)$-linear antisymmetric mapping of $Der(J)$ into $M,$ then $\Omega_{Der}(M)=\oplus\Omega^n_{Der}(M)$ is a module over $\Omega_{Der}(J)$ in the following way: for $\omega \in \Omega^n_{Der}(J)$ and $\Phi \in \Omega^l_{Der}(M),$ the action of $\omega$ on $\Phi$ is given by
\[
\left(\omega\Phi\right)\left(X_1,...,X_{n+l}\right)=
\frac{1}{(n+l)!}\sum_i (-1)^{\vert i\vert} \omega\left(X_{i_1},..., X_{i_n}\right)\Phi(X_{i_{n+1}},...,X_{i_{n+l}})
\]
where $i$ denotes a permutation of $\left(1,...,{n+l}\right)$ and $\mid i \mid$ denotes the parity of the permutation $i$.
 \begin{mydef}\label{sc}
Let $J$ be a Jordan algebra, let $M$ be a module over $J.$ A \textit{derivation-based connection} on $M$ is a linear endomorphism $\nabla$ of $\Omega_{Der}(M)$ such that
\begin{equation}\label{sdc}
\nabla(\Phi)\in \Omega^{l+1}_{Der}(M)
\end{equation}
and
\begin{equation} \label{ssc}
\nabla\left(\omega\phi\right)=d(\omega)\Phi+(-1)^n\omega \nabla \Phi .
\end{equation} 
for all $\omega \in \Omega^n_{Der}(J)$ and $\Phi \in \Omega^l_{Der}(M).$
\end{mydef} 
From \eqref{ssc} we see that if $\nabla$ and $\nabla'$ are two different connections, then their difference is an endomorphism of $\Omega_{Der}(M)$ as module over  $\Omega_{Der}(J).$
In this case the \textit{curvature} of a connection is defined as $R=\nabla^2.$
Definitions \eqref{pc} and \eqref{sc} are equivalent, in fact if $\nabla$ is a connection as in the second definition, one defines a map from $Der(J)$ into $End(M)$ by setting
\begin{equation}
\nabla_X(m)=(\nabla(m))(X)
\end{equation}
and the map $X \mapsto \nabla_X$ is a connection in the sense of \eqref{pc}. On the other hand, if $\nabla: X \mapsto \nabla_X$ is a connection according to the first definition, one sets
\begin{equation}
\begin{split}
\nabla(\Phi)\left(X_0,...,X_n\right)=&\sum _{0 \leq k \leq n}(-1)^k \nabla_{X_p} \left(\Phi\left(X_0,...,\widehat{X_{k}},... X_n\right)\right)+
\\& +\sum _{0 \leq r <s \leq n}(-1)^{r+s}  \Phi\left([X_r,X_s],X_0,...,\widehat{X_{r}},...,\widehat{X_{s}},... X_n\right) 
\end{split}
\end{equation}
for all $\Phi \in \Omega^n_{Der}(M)$ and $X_p \in Der(J)$ and $\nabla$ is now a connection according to definition \eqref{sc}.
\\In the following examples the term ``connection" will stand for derivation-based connection.
\begin{exa}
Let $J$ be a finite-dimensional and unital Jordan algebra, let $M=J \otimes E$ be a free $J$-module. On $M$ we have a base connection $\nabla^0=d\otimes Id_E:J\otimes E \rightarrow \Omega_{Der}^1\otimes E.$ As map from $Der(J)$ into $End(M),$ $\nabla^0$ is the lift of the differential on $J,$ that is 
\begin{equation}
\begin{split}
\nabla^0_X \left(x\otimes e\right)=\left(dx\right)(X) \otimes e 
\end{split}
\end{equation}
for any $ X \in Der(J)$ and $ \; x\otimes e \in M.$ 
It is easy to check that $\nabla^0$ respects properties \eqref{prc} and \eqref{src}. Moreover, this connection is gauge invariant whenever the center of $J$ is trivial.
$\diamond$
\begin{prop}\label{conn}
Let $J$ be a finite-dimensional Jordan algebra, let $M=J\otimes E$ be a free module over $J,$ where $E$ is a real vector space. Then any connection on $M$ is of the form
\begin{equation}
\nabla= \nabla^0 + \mathcal{A}
\end{equation}
where $\mathcal{A}$ is a linear map $\mathcal{A}: Der(J)\rightarrow Z(J) \otimes End(E)$ and  
\begin{equation}
\mathcal{A}(X)\left(x \otimes e\right)= x\otimes A(X)e 
\end{equation}
for all $ X\in Der(J)$ and $x \in J.$
\begin{proof}
From the definition of connection, it has to be
\begin{equation}
\nabla- \nabla^0=A \in End_J(M) 
\end{equation}
and from theorem \eqref{main}, it follows $A(X) \in Z(J) \otimes End(E).$
\end{proof}
\end{prop}
For what concerns flat connections, the following result, very similar to its counterpart in the context of Lie algebras, holds.
\begin{prop}
Let $M=J\otimes E$ be a free module over a simple Jordan algebra $J,$ then 
flat connections on $M$ are in one to one correspondence with Lie algebra homomorphisms $\mathcal{A}:Der(J)\rightarrow End(E).$ That is, for a basis $\{X_\mu\}\subset Der(J)$ with structure constants $c^\tau_{\mu \nu}$ one has
\begin{equation}\label{Liehomo}
\left[\mathcal{A}(X_\mu),\mathcal{A}(X_\nu)\right]=c^\tau_{\mu \nu} \mathcal{A}(X_\tau).
\end{equation}
where $\left[X_\mu, X_\nu\right]=c^\tau_{\mu \nu} X_\tau.$ 
\begin{proof}
By direct computation one can check that if a given connection $\nabla=\nabla^0+\mathcal{A}$ is flat then \eqref{Liehomo} must hold. 
\\On the converse, if $\mathcal{A}:Der(J)\rightarrow End(E)$ is such that \eqref{Liehomo} holds on a basis  $\{X_\mu\}\subset Der(J),$ then $\nabla=\nabla^0+\mathcal{A}$ is a flat connection on $M.$
\end{proof}
\end{prop}
Summarizing all the derivation-based differential calculus for free modules over Jordan algebras is resumed by the following proposition.
\begin{prop} Let $J$ be a unital Jordan algebra, let $M= J\otimes E$ be a free module over $J$ then
\begin{enumerate}
\item $\nabla^0= d\otimes \mathbb{I}_E:J\otimes E\rightarrow \Omega^1(J) \otimes E $ defines a flat connection on $M$ which is gauge-invariant whenever the center of $J$ is trivial.
\item Any other connection $\nabla$ on $M$ is defined by 
\begin{equation}
\nabla= \nabla^0 + A :J\otimes E\rightarrow \Omega^1(J) \otimes E 
\end{equation}
 where $A$ is a module homomorphism of $J\otimes E$ into $\Omega^1(J) \otimes E.$
\item For a derivation-based connection $\nabla$ the curvature is given by 
\begin{equation}
\nabla^2(X,Y)= R_{X,Y}= X(A(Y))-Y(A(X))+\left[A(X),A(Y)\right]-A([X,Y])
\end{equation} 
for any $X,Y\in Der(J).$
\item If $J$ is a simple Jordan algebra, then $\nabla$ defines a flat connection if and only if the map $A: Der(J)\rightarrow End(E)$ is a Lie algebra homomorphism.
\end{enumerate}
\end{prop}
\end{exa}
\begin{exa}
Consider again $A^i_n$ as module over $J^i_n.$ We can provide a base connection for this module. From \eqref{Jacobsonder} we know that for any $X \in Der(J^i_n)$ there exists a finite number of couples of $x_i, y_i \in J^i_n$ such that 
\begin{equation}\label{dam}
X(z)=\sum_i \left(x_i\circ z \right) \circ y_i - x_i\circ \left(j  \circ y_i \right)
\end{equation}
for any $z \in J^i_n$ and
where we have explicitly written $\circ$ to design matrix anticommutator. Let $X_i=\left[x_i,y_i\right],$ where the commutator is taken with respect to the standard row by column product, then the expression above can also be written as:
\begin{equation}
X(z)=\sum_i\left[X_i,z\right]
\end{equation}
for any $z \in J^i_n.$
\\Recall that the commutator of two hermitian matrices is an antihermitian matrix, then a good base connection on the Jordan module $A^i_n$ is given by
\begin{equation}
\nabla_X(a)=\sum_i\left[X_i,a\right]
\end{equation}
for all $a\in A^i_n,$ indeed:
\begin{equation}
\begin{split}
&\nabla_X(z\circ a)= \sum_i\left[X_i,z\circ a\right]=
\\&\left[X_i,z\right]\circ a+\sum_i\left[X_i,a\right]\circ z= X(z)\circ a + z\circ \nabla_X(a)
\end{split}
\end{equation}
for all $z \in J^i_n$ and $a \in A^i_n.$ 
Moreover this base connection is flat, indeed:
\begin{equation}
\begin{split}
&\left(\left[\nabla_x,\nabla_Y\right]-\nabla_{[X,Y]}\right)(a)=\left[X\left[Y,a\right]\right]-\left[Y\left[X,a\right]\right]-\left[\left[X,Y\right],a\right]=
\\&\left[\left[a,Y\right],X\right]+\left[\left[X,a\right],Y\right]+\left[\left[Y,X\right],a\right]=0
\end{split}
\end{equation}
in view of the Jacopi identity in the Lie algebra $M_n(\mathbb{R}). \quad \diamond$
\end{exa}
\begin{rema}
Due to commutativity, for any Jordan algebra $J$ it holds
\begin{equation}
[x,z,y]=-\left[L_x,L_y\right]z
\end{equation} 
for all $x,y,z \in J.$ Hence the commutator $\left[L_x,L_y\right]$ defines an inner derivation for $J.$ In fact, formula \eqref{dam} is a consequence of this in the particular case of special Jordan algebras.
\end{rema}
The example above can be generalized to the case of a module $M$ over any finite-dimensional, semisimple Jordan algebra. In fact, in view of theorem \eqref{Jacobsonder} all the derivations of such algebras are inner.
\begin{prop}
Let $J$ be a finite-dimensional semisimple Jordan algebra so that $\forall X \in Der(J)$ there exist a finite number of couples of elements $x_i,y_i\in J$ such that 
\begin{equation}
X(z)=\sum [x_i,z,y_i]
\end{equation}  
for every $z \in J.$
Then the map 
\begin{equation}
\begin{split} \label{css}
\nabla:& Der(J) \rightarrow End(M)
\\& X\mapsto \nabla_X=\sum [x_i,\cdot,y_i]
\end{split}
\end{equation}
is a connection on M.
\end{prop}
\begin{proof}
Let $X= \sum [x_i,\cdot,y_i] \in Der(J),$ it extends to a derivation  $\tilde{X}$ on the split null extension $J\oplus M$ given by
\begin{equation}
\tilde{X}(z,m)= \sum [(x_i,0),(z,m),(y_i,0)]
\end{equation}
for all $(z,m) \in J\oplus M.$
\\If we identify $M$ with elements of the form $(0,m)$ in $J\oplus M,$ we see that $\tilde{X}$ restricts to a linear endomorphism on $M.$ Then $\nabla$ is a $Z(J)$ linear map from $Der(J)$ into $End(M)$ and from Leibniz rule applied to $\tilde{X} \in Der (J \oplus M)$ we have
$\nabla_X(zm)= X(z)m + z\nabla_Xm.$ 
\end{proof}
 \begin{rema} Connection \eqref{css} can be defined for every Jordan module over a Jordan algebra for which all derivations are inner in the sense of theorem \eqref{Jacobsonder} and such that  the extension of derivations of the algebra to derivations on the split null extension is unique. The set of Jordan algebras for which all derivations are inner contains all finite-dimensional semi-simple Jordan algebra over a field of characteristic zero but it is in fact much wider, for example from theorem $2$ of \cite{Harris} we see that this request holds true for finite-dimensional and separable Jordan algebras on any field of characteristic different from $2.$
 \end{rema}
More generally we can give the following definition for a connection.
\begin{mydef} 
Let $\Omega=\oplus_{\mathbb{N}} \Omega^n$ be a differential graded Jordan algebra and let $\Gamma=\oplus_{\mathbb{N}} \Gamma^n$ be a graded Jordan module over $\Omega$, a \textit{connection} on $\Gamma$ is a linear endomorphism $\nabla:\Gamma\rightarrow \Gamma$ such that 
\begin{eqnarray}
&(\nabla \Phi)\in \Gamma^{l+1}
\\ &\nabla (\omega \Phi)= d(\omega) \Phi + (-1)^n \omega \nabla(\Phi)
\end{eqnarray}
for all $\omega\in \Omega^n$ and $\Phi \in \Gamma^l.$
\end{mydef}
In particular when $\Omega^0=J$ and $\Gamma^0=M$ one obtains the definition of $\Omega$-connection over the $J$-module $M.$


\begin{thebibliography}{9}
\bibitem{Albert}
A.A.Albert, \textit{On a Certain Algebra of Quantum Mechanic}, Annals of Mathematics, Second Series $\bf\/ 35,$  pp. $65-72,$ $1934.$
\bibitem{Baez}
J. Baez, \textit{The octonions}, arXiv:math/0105155, $2002$
\bibitem{C4}
A.Chamseddine, A.Connes,  \textit{The spectral action principle} Commun. Math. Phys. $\bf\/186,$ pp $731-750,$ $1997.$
\bibitem{C3}
A. Connes, \textit{Noncommutative differential geometry} Pub. IHES $\bf\/62,$ pp. $257-360,$ $1986.$
\bibitem{C2}
A. Connes, J. Lott \textit{Particles models and noncommutative geometry} Nucl. Phys $\bf\/ B18$ Suppl., pp. $29-47,$ $1990.$ 
\bibitem{C1}
A. Connes,\textit{Noncommutative geometry}, Accademic Press, $1994.$
\bibitem{CM}
A. Connes, M. Marcolli \textit{Noncommutative geometry, quantum fields and motives} American Mathematical Society, $2008.$
\bibitem{DVKM} 
 M. Dubois-Violette, R. Kerner, J. Madore, \textit{Gauge bosons in a non-commutative geometry}, Phys. Lett. $\bf\/ B 217,$ pp. $485-488,$ $1989.$
\bibitem{MDV2}
M. Dubois-Violette  \textit{Noncommutative differential geometry, quantum mechanics and gauge theory.} In: Bartocci C., Bruzzo U., Cianci R. (eds) \textit{Differential Geometric Methods in Theoretical Physics.} Lecture Notes in Physics $\bf\/ 375$ Springer, $1991.$
\bibitem{MDV3}
M. Dubois-Violette \textit{Lectures on graded differential geometry and noncommutative geometry} in: Y. Maeda et al. \textit{Noncommutative differential geometry and its applications to physics}, Shonan, Japan, pp. $245-306$, Kluwer accademic pubblishers, $2001.$ 
\bibitem{MDV1} 
M. Dubois-Violette
\textit{Exceptional quantum geometry and particle physics}
Nuclear Physics {\bf\/ B912}, pp.  $426-449,$ $2016.$
\bibitem{Ei}
S. Eilenberg \textit{Extensions of general algebras} Ann. Soc. Math. Pol. $\bf\/ 21,$ pp. $125-134,$ $1948.$
\bibitem{GPR}1
M. Gunaydin, C. Piron, H. Ruegg, \textit{Moufang plane and octionic quantum mechanics}, Commun. Math. Phys. $\bf\/ 61,$ pp. $69-85,$ $1978.$
\bibitem{Harris}
B. Harris,\textit{Derivations of Jordan algebras} Pacific J. Math. {\bf\/ 9},  pp. $495-512$, $1959.$
\bibitem{Hue}
J. Huebschmann \textit{Lie-Rineheart algebras. Gerstenhaber algebras and Batalin-Vilkovisky algebras}  Ann. Inst. Fourier Grenoble  {\bf\/ 48}, pp. $425-440,$ $1998.$ 
\bibitem{Iordanescu}
R. Iordanescu \textit{Jordan structures in analysis, geometry and physics}, Ed. Acad. Romane, $2009.$
\bibitem{Jacobson2}
N. Jacobson \textit{General representation theory of Jordan algebras}, Trans. Amer. Math. Soc. {\bf\/ 70},  pp.
$509-530,$ $1951.$
\bibitem{Jacobson1}
N. Jacobson, \textit{Structure and Representations of Jordan Algebras}, American Mathematical Society, $1968.$.
\bibitem{JNW}
 P. Jordan , J. von Neumann, E. Wigner, \textit{On an Algebraic Generalization of the Quantum Mechanical Formalism}, Annals of Mathematics,  Princeton, {\bf\/ 35}, pp. $29-64,$ $1934.$
\bibitem{Kac}
V.G. Kac \textit{Classification of simple Z-graded Lie superalgebras and simple Jordan superalgebras}, Commun. Algebra {\bf\/ 5},  pp. $1375-1400$ $1977.$
\bibitem{KOS}
I. Kashuba, S. Osvienko, I. Shestakov, \textit{Representation type of Jordan algebras}, Adv. Math. {\bf\/ 226}, pp. $1-35,$ $2011.$
\bibitem{Koszul}
J.L. Koszul,  \textit{Fibre bundles and diferential geometry}, Tata Institute of fundamental research, Bombay $1960.$
\bibitem{McCrimmon}
K. McCrimmon,  \textit{A taste of Jordan algebras}, Springer, $2004$
\bibitem{Rin}
G. Rinehart,  \textit{Differential forms for general commutative algebras} Trans. Amer. Math. Soc. {\bf\/ 108}, pp. $195-222,$ $1963.$
\bibitem{vS}
W. van Suijlekom, \textit{Noncommutative Geometry and Particle Physics}, Springer, $2014.$
\bibitem{Woronowicz}
S. L. Woronowicz,  \textit{Differential calculus on compact matrix pseudogroups (quantum groups)}, Commun. Math. Phys. {\bf\/ 122}, pp. $125-170,$ $1989.$
\bibitem{Wu}
A. Wulfsohn, \textit{Tensor product of Jordan algebras}. Can. J. Math.  Vol. {\bf\/ XXVIII},  $1975.$
\bibitem{Yokota} 
I. Yokota, 
\textit{Exceptional Lie Groups}
arXiv:0902.0431, $2009.$
\bibitem{ZS}
V. N. Zhelyabin and I. P. Shestakov,  \textit{Simple special Jordan superalgebras with associative even part.} Siberian Mathematical Journal {\bf\/ 45},  pp. $860-882,$ $2004.$
\end{thebibliography}
\end{document}